\documentclass[12pt]{amsart}
\usepackage{amsmath}
\usepackage{amsthm}
\usepackage{amsfonts}
\usepackage{amssymb}
\newtheorem{Theorem}{Theorem}[section]
\newtheorem{Proposition}[Theorem]{Proposition}
\newtheorem{Lemma}[Theorem]{Lemma}

\theoremstyle{definition}

\newtheorem{Hypothesis}{Hypothesis}[section]
\theoremstyle{remark}

\numberwithin{equation}{section}

\newcommand{\Z}{{\mathbb Z}}
\newcommand{\R}{{\mathbb R}}
\newcommand{\C}{{\mathbb C}}

\begin{document}

\title[Reflectionless operators]{The Marchenko representation of reflectionless Jacobi and Schr\"odinger operators}
\author{Injo Hur}

\address{Mathematics Department\\
University of Oklahoma\\
Norman, OK 73019}

\email{ihur@math.ou.edu}

\author{Matt McBride}

\address{Mathematics Department\\
University of Oklahoma\\
Norman, OK 73019}

\email{mmcbride@math.ou.edu}

\author{Christian Remling}

\address{Mathematics Department\\
University of Oklahoma\\
Norman, OK 73019}

\email{cremling@math.ou.edu}

\urladdr{www.math.ou.edu/$\sim$cremling}

\date{January 28, 2014}



\thanks{CR's work has been supported
by NSF grant DMS 0758594}
\begin{abstract}
We consider Jacobi matrices and Schr\"odinger operators that are reflectionless on an interval.
We give a systematic development of a certain parametrization of this class, in terms of suitable spectral data,
that is due to Marchenko. Then some applications of these ideas are discussed.
\end{abstract}
\maketitle
\section{Introduction}
We are interested in one-dimensional Schr\"odinger operators,
\begin{equation}
\label{se}
(Hy)(x) = -y''(x) + V(x) y(x) ,
\end{equation}
with locally integrable potentials $V$ that are in the limit point case at $\pm\infty$
and in Jacobi matrices,
\begin{equation}
\label{jac}
(Ju)_n = a_n u_{n+1} + a_{n-1}u_{n-1} + b_n u_n .
\end{equation}
Here we assume that $a,b\in\ell^{\infty}(\Z)$, $a_n >0$, $b_n\in\R$.

These operators have associated half line $m$ functions $m_{\pm}$. These are \textit{Herglotz functions,}
that is, they map the upper half plane $\C^+$ holomorphically to itself. (The precise definitions of $m_{\pm}$
for Schr\"odinger and Jacobi operators will be reviewed below.)

Recall that we call an operator \textit{reflectionless }on a Borel set $S\subset\R$ of positive Lebesgue measure
if $m_{\pm}$ satisfy the following identity
\begin{equation}
\label{defrefl}
m_+(x) = -\overline{m_-(x)} \quad\quad \textrm{for (Lebesgue) a.e.\ } x\in S .
\end{equation}
Reflectionless operators are important because they can be thought
of as the fundamental building blocks of arbitrary operators with some absolutely continuous
spectrum. See \cite{Kotac,Remcont,Remac}.
Reflectionless operators have remarkable properties, and if an operator is
reflectionless on an \textit{interval }(rather than a more complicated set), one can say
even more. So these operators are of special interest.

Marchenko \cite{Mar} developed a certain parametrization of the class $\mathcal M_R$ of Schr\"odinger operators
$H$ that are reflectionless on $(0,\infty)$ and have spectrum contained in $[-R^2,\infty)$ (we are paraphrasing,
Marchenko does not emphasize this aspect, and his goals are different from ours).
It is in fact easy in principle to give such a parametrization in terms of certain spectral data, which
has been used by many authors \cite{Craig,PR1,PR2,Remac,Remtop,Teschl}. We will briefly
review this material in Sect.~2. Marchenko's parametrization is different, and it makes
certain properties of reflectionless Schr\"odinger operators very transparent. Some of these applications will be discussed below.
For a rather different application of Marchenko's parametrization, see \cite{KotKdV}, where this material
is used to construct the KdV flow on $\mathcal M_R$.

We have two general goals in this paper. First, we present a direct and easy
approach to Marchenko's parametrization that starts from scratch and does not use any machinery.
Marchenko's treatment relies on inverse scattering theory as its main
tool (which then needs to be combined with a limiting process, as most reflectionless operators do not fall under
the scope of classical scattering theory) and is rather intricate. 
We hope that our approach will help put things in their proper context;
among other things, it will explain the role of
the inequalities imposed on the representing measures $\sigma$.
We will also extend these ideas to the discrete setting; in fact, we will start with this case as some technical issues
from the continuous setting are absent here.
The second goal is to explore some consequences and applications of Marchenko's parametrization, in the form
developed here. We will have more to say about this towards the end of this introduction.

For the precise statement of the Marchenko parametrizations, we refer the reader to Theorems \ref{Tdisc}, \ref{Tcont} below.
However, the basic ideas are easy to describe. If $S$ is an interval, then it's well known (compare, for example,
\cite[Corollary 2]{Kot85}) that \eqref{defrefl}
guarantees the existence of a genuine holomorphic continuation of $m_+$ through $S$ (this is not an immediate
consequence of the Schwarz reflection principle because of the possible presence of an exceptional set where \eqref{defrefl} fails).
More precisely, we have the following (the proof will be reviewed in Sect.\ 2).
\begin{Lemma}
\label{L2.1}
Fix an open interval $S=(a,b)$, and let $m_+$ be a Herglotz function. Then
$m_+$ satisfies \eqref{defrefl} for $S=(a,b)$ (for some Herglotz function $m_-$) if and only if $m_+$ has a holomorphic continuation
\[
M: \C^+\cup S \cup \C^-\to\C^+ .
\]
\end{Lemma}
Note that there are two conditions really: $m_+$ must have a continuation $M$ to $\Omega=\C^+\cup S\cup\C^-$,
and, moreover, $M$ must map all of $\Omega$ to $\C^+$.
However, these properties are immediate consequences
of the fact that if $S=(a,b)$, then the exceptional null set from \eqref{defrefl} is empty, so this is what the Lemma really says.

This continuation $M$ is necessarily given by $M(z)=-\overline{m_-(\overline{z})}$ on the lower half plane $z\in\C^-$. In other
words, \eqref{defrefl} for $S=(a,b)$ lets us combine $m_+$ and $m_-$ into one holomorphic function $M$ on the
simply connected domain $\Omega$. We can then introduce a conformal change of variable $z=\varphi(\lambda)$,
$\varphi:\C^+\to\Omega$, to obtain a new Herglotz function $F(\lambda)\equiv M(\varphi(\lambda))$. The measures from the
Herglotz representations of these functions $F$ will be the data that we will use to
parametrize the operators from the Marchenko class $\mathcal M_R$.

Let us now discuss some applications.
As an immediate minor pay-off, we obtain a very quick new proof of \cite[Theorem 1.2]{RemDR}, which is now seen to be
an immediate consequence of our Theorem \ref{Tdisc} below. Recall that this result states that if a Jacobi matrix is bounded and
reflectionless on $(-2,2)$, then $a_n\ge 1$ for all $n\in\Z$, and if $a_{n_0}=1$ for a single $n_0\in\Z$, then $a_n\equiv 1$, $b_n\equiv 0$.
In Proposition \ref{P3.11} we try to indicate how these ideas could, perhaps, be carried further.

More importantly,
the material from Sect.~4 yields continuous analogs of these results. Here are three such
consequences of the Marchenko parametrization, combined with the material from \cite{Remcont}.
We are now interested in half line Schr\"odinger operators $H_+$ on $L^2(0,\infty)$ satisfying the following assumptions:
\begin{Hypothesis}
\label{H1.1}
$\Sigma_{ac}(H_+) \supset (0,\infty)$ and $V$ is uniformly locally integrable, that is,
\begin{equation}
\label{Vlocint}
\sup_{x\ge 0} \int_x^{x+1} |V(t)|\, dt < \infty .
\end{equation}
\end{Hypothesis}
Here, $\Sigma_{ac}$ denotes an essential support of the absolutely continuous part of the spectral measure
of $H_+$. In other words, we are assuming that $\chi_{(0,\infty)}(E)\, dE\ll  d\rho_{ac}(E)$. This implies that, but is not equivalent to
$\sigma_{ac}(H_+)\supset [0,\infty)$.
An $H_+$ satisfying Hypothesis \ref{H1.1} can, of course, have embedded singular spectrum in $(0,\infty)$, and can have
arbitrary spectrum outside this set.

Notice also that \eqref{Vlocint} implies that $H_+$ is bounded below.

To obtain self-adjoint operators, one has to impose a boundary condition at $x=0$, but since $\Sigma_{ac}$
is independent of this boundary condition, we won't make it explicit here.

Let us now state three (closely related) sample results.
\begin{Theorem}
\label{T1.1}
Assume Hypothesis \ref{H1.1}. Then
\begin{equation}
\label{1.6}
\limsup_{x\to\infty} \int V(x+t)\varphi (t) \, dt \le 0
\end{equation}
for every compactly supported, continuous function $\varphi\ge 0$.
\end{Theorem}
This says that in the situation described by Hypothesis \ref{H1.1}, the positive part of $V$ will go to zero, in a weak sense.
\begin{Theorem}
\label{T1.1b}
Assume Hypothesis \ref{H1.1}. If, in addition, $V\ge 0$ on $\bigcup (x_n-d,x_n+d)$ for some
increasing sequence $x_n\to\infty$ with bounded gaps
(that is, $\sup (x_{n+1}-x_n)<\infty$) and some $d>0$, then
\begin{equation}
\label{1.21}
\lim_{x\to\infty} \int V(x+t)\varphi(t)\, dt = 0
\end{equation}
for every compactly supported, continuous function $\varphi$.
\end{Theorem}
Theorem \ref{T1.1b} is a variation on the (continuous) Denisov-Rakhmanov Theorem \cite{Den,Remcont}. Recall that the DR Theorem asserts that
\eqref{1.21} will follow
if, in addition to Hypothesis \ref{H1.1}, we have that $\sigma_{ess}(H_+)=[0,\infty)$,
In Theorem \ref{T1.1b}, we replace this latter assumption by partial information on $V$; more precisely,
we assume here that $V$ is non-negative every once in a while,
with positive frequency.
\begin{Theorem}
\label{T1.1c}
Assume Hypothesis \ref{H1.1}. We are given $d>0$ (arbitrarily small) and $\epsilon>0$
and (arbitrarily many) compactly supported, continuous test functions
$\varphi_1,\ldots,\varphi_N$.
Then there exist $x_0>0$ and $\delta>0$ so that the following holds: If $x\ge x_0$ and $V(t)\ge -\delta$ for $|t-x|< d$, then
\[
\left| \int V(t)\varphi_j(t-x) \, dt \right| < \epsilon
\]
for $j=1,\ldots, N$.
\end{Theorem}
In particular, this conclusion is obtained if $V\ge 0$ on $|t-x|<d$, in which case $\delta$ becomes irrelevant.

This is an Oracle Theorem type statement that, roughly speaking, says that if $V$ is almost non-negative
anywhere, then $V$ has to be close to zero on a very long interval centered at that point (not in a pointwise sense, though).

Let us now discuss a completely different application of the Marchenko parametrization.
Call a half line operator $H_+$ or $J_+$ (on $L^2(0,\infty)$ or $\ell^2(\Z_+)$, respectively) reflectionless on $S$ if the
corresponding $m$ function $m_+$ satisfies \eqref{defrefl} for some (unique, if it exists at all) Herglotz function $m_-$.

Reflectionless half line operators may, of course, be obtained by restricting reflectionless whole line problems. Since
reflectionless operators may be reconstructed from arbitrary half line restrictions, we can actually think of such a half
line restriction as just another representation of the original whole line problem. Perhaps
somewhat surprisingly, however, there are other examples:
\begin{Theorem}
\label{T1.2}
(a) There exists a half line Jacobi matrix $J_+$ that is reflectionless on $(-2,2)$, but is not the restriction of a reflectionless whole line
Jacobi matrix.\\
(b) There exists a half line Schr\"odinger operator $H_+$ that is reflectionless on $(0,\infty)$, but is not the restriction of a reflectionless
whole line Schr\"odinger operator.
\end{Theorem}
Put differently, the associated $m$ function $m_-$ that is obtained from $m_+$ via \eqref{defrefl} is not the $m$ function of
a Jacobi matrix or Schr\"odinger operator, respectively. The examples we will construct to prove Theorem \ref{T1.2} will be quite explicit,
especially in the discrete case;
they will satisfy $\sigma(J_+)=[-2,2]$, $\sigma(H_+)=[0,\infty)$, so it is not spectrum outside $S$ (there isn't any) that
produces this effect. We will see below that Theorem \ref{T1.2} is in fact a rather quick consequence of the Marchenko parametrization.

We organize this paper as follows. Section~2 presents a very quick review of certain spectral data that are tailor made for the
discussion of reflectionless operators; we also prove Lemma \ref{L2.1} there. In Sections 3 and 4, we formulate and prove the
Marchenko parametrizations in the discrete and continuous settings, respectively. The remaining results are proved in the final
two sections.
\section{Preliminaries}
We briefly review some standard material about certain spectral data that are particularly convenient
if one wants to discuss reflectionless operators. See \cite{PR2,Remac} for a more comprehensive discussion.

Given a pair of Herglotz functions $m_{\pm}$ that satisfies \eqref{defrefl}, consider
$H=m_+ +m_-$. Since this is another Herglotz function, we can take a holomorphic logarithm,
which is a Herglotz function itself, if we agree that $\textrm{Im}\,\ln H\in (0,\pi)$, say.  The \textit{Krein function }of $H$ is then defined
(almost everywhere, with respect to Lebesgue measure) by
\[
\xi(x) = \frac{1}{\pi} \lim_{y\to 0+} \textrm{Im}\,\ln H(x+iy) .
\]
We have that $0\le\xi\le 1$, and \eqref{defrefl} implies that
$\xi=1/2$ a.e.\ on $S$. Next, if
\[
H(z) = A + Bz + \int_{-\infty}^{\infty} \left(\frac{1}{t-z} - \frac{t}{t^2+1} \right)\, d\rho (t)
\]
is the Herglotz representation of $H$, then it's easy to verify (see, for example, \cite[Sect. 5]{Remac} for the details) that
\begin{equation}
\label{2.1}
m_+(z) = A_+ + B_+z + \int_{-\infty}^{\infty} \left(\frac{1}{t-z} - \frac{t}{t^2+1} \right)f(t) \, d\rho (t) ,
\end{equation}
and here $0\le B_+\le B$, $0\le f\le 1$, $f=1/2$ Lebesgue-a.e.\ on $S$.

Conversely, these data determine an $m_+$ that will satisfy \eqref{defrefl}.
More explicitly, if measurable functions $\xi, f$ with $0\le \xi, f\le 1$
and $\xi=f=1/2$ a.e.\ on $S$ are given, and if we also choose three constants $C>0$, $0\le c\le 1$, $A_+\in\R$, then $\xi$ and $C$ first of all
determine a unique $H$ with $|H(i)|=C$. We in fact have the explicit formula
\begin{equation}
\label{expH}
H(z) = C\exp \left[ \int_{-\infty}^{\infty} \left(\frac{1}{t-z}-\frac{t}{t^2+1}\right)\xi(t)\, dt \right] .
\end{equation}
Then \eqref{2.1} with $B_+=cB$ defines an $m_+$, which will satisfy \eqref{defrefl},
with $m_-=H-m_+$. Any $m_+$ satisfying \eqref{defrefl} is obtained in this way.

Let us now sketch the proof of Lemma \ref{L2.1}.
\begin{proof}[Proof of Lemma \ref{L2.1}]
Obviously, if $M$ is as in the lemma, then \eqref{defrefl} holds, with $m_-(z):=-\overline{M(\overline{z})}$ ($z\in\C^+$).

Conversely, assume that \eqref{defrefl} holds with $S=(a,b)$. Since it suffices to prove the claim for arbitrary bounded
subintervals of $S$, we may assume that $S$ itself is bounded. Now consider $H$, defined as above.
As observed earlier, its Krein function satisfies $\xi=1/2$ a.e.\ on $S$. Since
\[
\frac{1}{2} \int_a^b \frac{dt}{t-z} ,
\]
originally defined for $z\in\C^+$,
has a holomorphic continuation through $(a,b)$ (evaluate the integral!), we see from the exponential Herglotz representation \eqref{expH}
that $H$ itself has the same
property. Now \eqref{2.1} makes it clear that $m_+$ has such a holomorphic continuation, too. Here we use the fact
that in the situation under consideration, $\rho$ cannot have a singular part on $(a,b)$; this follows immediately from our earlier
observation that $H$ can be holomorphically continued through this interval.

By \eqref{defrefl}, this continuation of $m_+$ must be
given by $M(z)=-\overline{m_-(\overline{z})}$ for $z=x-iy$, $a<x<b$, $y>0$ and small, so we can actually continue to all of $\C^-$ and
this continuation clearly maps $\C^+\cup\C^-$ to $\C^+$, and $\textrm{Im}\, M(x)\ge 0$ for $x\in S$. The proof is now finished by
observing that the open mapping theorem gives us strict inequality here.
\end{proof}
\section{The discrete case}
We are now interested in Jacobi matrices $J$ on $\ell^2(\Z)$ that are reflectionless on $S=(-2,2)$ and satisfy $\|J\|\le R$
for some $R\ge 2$. We will denote the collection of these Jacobi matrices by $\mathcal M_R$.

The half line $m$ functions may be defined as follows: For $z\in\C^+$, let $f_{\pm}(\cdot, z)$
be the solutions of
\[
a_n f(n+1,z) + a_{n-1} f(n-1,z) + b_n f(n,z) = zf(n,z)
\]
that are square summable near $\pm\infty$ (the assumption that $J$ is bounded makes sure that these are unique up to multiplicative constants).
Now we can let
\[
m_{\pm}(z) = \mp \frac{f_{\pm}(1,z)}{a_0 f_{\pm}(0,z)} .
\]
We are assuming \eqref{defrefl} on $(-2,2)$, so
by Lemma \ref{L2.1}, we can combine $m_{\pm}$ into one function $M:\Omega \to \C^+$, $\Omega=\C^+\cup (-2,2)\cup\C^-$.
Off the interval $(-2,2)$, $M$ is given by
\begin{equation}
\label{formM}
M(z) = \begin{cases} m_+(z) & z\in\C^+ \\ -\overline{m_-(\overline{z})} & z\in\C^- \end{cases} .
\end{equation}
Following our earlier outline, we now want to introduce a conformal change of variable $\varphi: \C^+\to\Omega$. We will work
with the specific map\[
\varphi(\lambda)=-\lambda-\frac{1}{\lambda} .
\]
In the subsequent developments, it is useful to keep in mind that
$\varphi$ maps the upper half of the unit circle onto $(-2,2)$. The upper semi-disk is mapped onto $\C^+$, while the complement (in $\C^+$,
of the closed disk) goes to $\C^-$ under $\varphi$. (Of course, $\varphi$ is defined by the formula given for arbitrary $\lambda\not= 0$, and we
will frequently make use of this extended map without further comment.)

As anticipated, we now define the new Herglotz function
\[
F(\lambda) = M(\varphi(\lambda)) \quad\quad (\lambda\in\C^+) .
\]
It will also be convenient to let $r$ denote the solution $r+1/r=R$ with $0<r\le 1$; this is well defined because we are assuming that $R\ge 2$.
Also, we will write $\sigma_n \equiv \int t^n\, d\sigma(t)$ for the (generalized) moments of a measure $\sigma$, for $n\in\Z$.
\begin{Theorem}
\label{Tdisc}
$J\in\mathcal M_R$ if and only if the associated $F$ function is of the form
\begin{equation}
\label{3.2}
F(\lambda) = -\sigma_{-1} + (1-\sigma_{-2})\lambda + \int \frac{d\sigma(t)}{t-\lambda} ,
\end{equation}
for some finite Borel measure $\sigma$ on $(-1/r, -r)\cup (r,1/r)$ that satisfies
\begin{equation}
\label{3.1}
1-\sigma_{-2} + \int \frac{d\sigma(t)}{t^2+Et+1} > 0
\end{equation}
for all $|E|>R$.
\end{Theorem}
To spell this out even more explicitly, this says that if $J\in\mathcal M_R$, then the associated $F$ will have a representation of
the form \eqref{3.2}, with a $\sigma$ that has the stated properties. It is also clear that we have uniqueness: $J$ determines
$m_{\pm}$ and thus $F$ and $\sigma$.
Conversely, if a measure $\sigma$ satisfies \eqref{3.1} (and is supported on the set given), then \eqref{3.2} defines a function that is the
$F$ function of a unique $J\in\mathcal M_R$.

In other words, Theorem \ref{Tdisc}
sets up a one-to-one correspondence between $J\in\mathcal M_R$ and the measures $\sigma$ on $r< |t|< 1/r$ satisfying \eqref{3.1}.

If we are not interested in the specific value of $\|J\|$, then we may interpret Theorem \ref{Tdisc} as setting up a one-to-one correspondence
between bounded, reflectionless (on $(-2,2)$) Jacobi matrices and measures $\sigma$ that are supported by a compact subset of $\R\setminus
\{ 0\}$ and satisfy $\sigma_{-2}<1$. To obtain this version, it suffices to observe that the integral from \eqref{3.1} goes to zero as $|E|\to\infty$.

The proof will depend on the asymptotic properties of $m_{\pm}$ for a Jacobi matrix, so we briefly review these first.
See, for example, \cite[Ch.~2]{Teschl} for this material.

For any $J$ with $\|J\|\le R$, we have that
\begin{gather}
\label{m+}
m_+(z) = \int \frac{d\rho_+(t)}{t-z} \\
\label{m-}
a_0^2 m_-(z) = z-b_0 + a_{-1}^2 \int \frac{d\rho_-(t)}{t-z} ,
\end{gather}
and here $\rho_{\pm}$ are probability (Borel) measures supported by $[-R,R]$. Conversely,
if we are given such data (that is, we are given two compactly supported probability measures $\rho_{\pm}$ and numbers $a_0,a_{-1}>0$,
$b_0\in\R$), then there will be a bounded whole line Jacobi matrix $J$ with half line $m$ functions given by \eqref{m+}, \eqref{m-}.
Moreover, if both $\rho_+$ and $\rho_-$ have infinite supports, then this $J$ will be unique.

Whether or not a given Herglotz function has a representation of this type can be decided by looking at the large $z$ asymptotics:
\begin{Lemma}
\label{L3.2}
Let $g$ be a Herglotz function and let $a>0$. Then
\[
g(z)=\int \frac{d\rho(t)}{t-z}, \quad\quad \rho(\R)=a
\]
for some finite measure $\rho$ if and only if
$\lim_{y\to\infty}yg(iy)=ia$.
\end{Lemma}
\begin{proof}
If $g$ has such a representation, then $yg(iy)\to ia$ follows immediately from dominated convergence. To prove the converse, write down
the (general) Herglotz representation of $g$:
\begin{equation}
\label{3.3}
g(z) = A + Bz + \int \left(\frac{1}{t-z}-\frac{t}{t^2+1} \right)\, d\rho(t)
\end{equation}
Then
\[
y\,\textrm{Im}\,g(iy) = By^2 + \int \frac{y^2}{t^2+y^2}\, d\rho(t) .
\]
By monotone convergence, the integral converges to $\rho(\R)$, so it follows that $\rho(\R)=a$ and $B=0$. In particular, we know now
that $\rho$ is finite, so we may split the integral from \eqref{3.3} into two parts and, using the hypothesis again, we then conclude that
$A-\int t/(t^2+1)\, d\rho(t)=0$.
\end{proof}
We are now ready for the
\begin{proof}[Proof of Theorem \ref{Tdisc}]
We first show that $F,\sigma$ have the asserted properties if $J\in\mathcal M_R$. Recall first of all that $m_{\pm}$ have holomorphic
continuations to a neighborhood of $(-\infty,-R)\cup (R,\infty)$. (This continuation of $m_+$ will, of course, be different from the continuation $M$ of the
same function, where the domains overlap.)
This follows because $\rho_{\pm}$ are supported by $[-R,R]$.
As a consequence, $F$ can be holomorphically continued through $\R\setminus \{ t: r\le |t| \le 1/r \}$; indeed, the set removed
contains all those $t\in\R$ that get mapped to $[-R,R]$ under the map $\varphi$.
At $t=0$, we need to argue slightly differently: $F$ can be holomorphically to a neighborhood of this point
because $m_+(z)$ is holomorphic at $z=\infty$. We will discuss this in more detail shortly.

So, if we now write down the Herglotz
representation of $F$, then the representing measure $\sigma$ will be supported by $\{t: r\le |t|\le 1/r \}$.
In particular, such a $\sigma$ is finite,
so we may again split off the $t/(t^2+1)$ term in \eqref{3.3} and absorb it by $A$. We arrive at the following representation:
\begin{equation}
\label{3.98}
F(\lambda) = A + B\lambda + \int\frac{d\sigma(t)}{t-\lambda}
\end{equation}

We can now identify $A,B$ by comparing the asymptotics of this function, as $\lambda\to 0$, with those of $m_+$.
Indeed, if $\lambda\in\C^+$ is close to zero, then $\varphi(\lambda)\in\C^+$, so $F(\lambda)=m_+(\varphi(\lambda))$ for these $\lambda$,
and \eqref{m+} shows that
\[
m_+(\varphi(\lambda)) =  -\frac{1}{\varphi(\lambda)} + O(\lambda^2) = \lambda + O(\lambda^2) .
\]
This confirms that $\sigma(\{ 0\})=0$, as claimed earlier.
We then see from a Taylor expansion of \eqref{3.98} that
\[
F(\lambda) = A + \sigma_{-1} + (B+\sigma_{-2})\lambda + O(\lambda^2) .
\]
It follows that $A=-\sigma_{-1}$ and $B=1-\sigma_{-2}$, as asserted in \eqref{3.2}.

To obtain \eqref{3.1}, we take a look at the function $H(z)=m_+(z)+m_-(z)$. As observed above, in the proof of Lemma \ref{L2.1},
$H$ has a holomorphic continuation through $(-2,2)$. Equivalently, the function $h(\lambda)=H(\varphi(\lambda))$, originally defined
for $\lambda\in\C^+$, $|\lambda|<1$, may be holomorphically continued through the upper half of the unit circle. On $|\lambda|=1$,
we can obtain this continuation as
\[
h(\lambda) = F(\lambda)-\overline{F(\lambda)}
\]
and since $\overline{\lambda}=1/\lambda$ for these $\lambda$, this gives
\begin{align}
\nonumber
h(\lambda) & = B\left(\lambda-\frac{1}{\lambda}\right) +\int\left( \frac{1}{t-\lambda}-\frac{1}{t-1/\lambda}\right)\, d\sigma(t)\\
& = \left(\lambda-\frac{1}{\lambda}\right) \left( 1-\sigma_{-2} + \int \frac{d\sigma(t)}{(t-\lambda)(t-1/\lambda)} \right ) .
\label{3.5}
\end{align}
Since the right-hand sides are analytic functions of $\lambda$, these formulae hold for all $\lambda\in\C^+$, $|\lambda|\le 1$.
It is useful to observe here that $h_0=\lambda-1/\lambda$ is the $H$ function of the free Jacobi matrix $a_n\equiv 1$, $b_n\equiv 0$.
Now $a_0^2 H(z) = -1/g(z)$, where $g(z)=\langle \delta_0, (J-z)^{-1}\delta_0 \rangle$ is the Green function of $J$ at $n=0$. This
implies that $H(x)<0$ for $x<-R$ (to the left of the spectrum) and $H(x)>0$ for $x>R$. Since $h_0$ already has the correct signs,
this forces the last factor from \eqref{3.5} to be positive for $|E|>R$. This gives \eqref{3.1}.

Finally, observe that \eqref{3.1} also prevents point masses at $t=\pm r$, $t=\pm 1/r$, so $\sigma$ is indeed supported by the
(open) set given in the Theorem.
For example, if we had $\sigma(\{ r\})>0$, then the integral from \eqref{3.1} would diverge to $-\infty$ as $E\to -R$, $E<-R$.

Conversely, assume now that a measure $\sigma$ on $(-1/r,-r)\cup (r,1/r)$ satisfying \eqref{3.1} is given.
We want to produce a $J\in\mathcal M_R$ so that
this $\sigma$ represents its $F$ function. It is clear how to proceed: define $F$ by \eqref{3.2} and let
\begin{align}
\label{m++}
m_+(\varphi(\lambda)) &= F(\lambda) \quad & (|\lambda|<1, \lambda\in\C^+) , \\
m_-(\varphi(\lambda)) &= -\overline{F(\overline{\lambda})} \quad & (|\lambda|>1, \lambda\in\C^-) .
\label{m--}
\end{align}
Since $\varphi$ maps both of these domains conformally onto $\C^+$, this defines two Herglotz functions $m_{\pm}$. As the first step,
to just obtain a Jacobi matrix $J$ from $m_{\pm}$, we have to verify
that these functions satisfy \eqref{m+}, \eqref{m-}.

So let $y>0$ (typically large), and let $s>0$ be the unique positive solution of $1/s-s=y$. Then $\varphi(is)=iy$ and $s=1/y+O(1/y^3)$.
Thus a Taylor expansion of \eqref{3.2} shows that $m_+$, defined by \eqref{m++}, satisfies
$m_+(iy) = i/y + O(y^{-2})$. Lemma \ref{L3.2} implies that $m_+$ satisfies \eqref{m+}, with $\rho_+(\R)=1$. In fact, $\rho_+$
is supported by $[-R,R]$. This follows because the definition \eqref{m++} also makes sure that $m_+(z)$ can be holomorphically
continued through the complement (in $\R$) of this interval.

Similarly, for large positive $t$, we have that
\[
-\overline{F(it)} = i(1-\sigma_{-2}) t + \sigma_{-1} + \frac{i\sigma_0}{t} + O(t^{-2}) .
\]
As before, take $t>1$ to be the solution of $\varphi(it)=iy$ for (large) $y>0$.
It then follows that $m_-$, defined by \eqref{m--}, satisfies
\begin{equation}
\label{3.12}
m_-(iy) = i(1-\sigma_{-2})y +\sigma_{-1} + i\frac{1-\sigma_{-2}+\sigma_0}{y} + O(y^{-2}) \quad\quad (y\to\infty) .
\end{equation}
We can now again refer to Lemma \ref{L3.2} to conclude that $m_-$ satisfies \eqref{m-}, with
\begin{equation}
\label{a0}
a_0 = \left(1-\sigma_{-2}\right)^{-1/2} , \quad b_0 = -\frac{\sigma_{-1}}{1-\sigma_{-2}} .
\end{equation}
Note in this context that \eqref{3.1} implies that $1-\sigma_{-2}>0$. So \eqref{a0} does
define coefficients $a_0>0$, $b_0\in\R$. By suitably defining $a_{-1}>0$, we can then guarantee that
$\rho_-(\R)=1$. As above, we also see that $\rho_-$ is in
fact supported by $[-R,R]$.

By the material reviewed at the beginning of this section, we obtain a unique Jacobi matrix $J$ from the pair $m_{\pm}$.
It is indeed unique because $\rho_{\pm}$ are equivalent to Lebesgue measure on $(-2,2)$, so are certainly not supported by
a finite set. It is immediate from the definition of $m_{\pm}$ that this $J$ will be reflectionless on $(-2,2)$, and, by construction,
its $F$ function is represented by the measure $\sigma$ we started out with.

It remains to show that $\|J\|\le R$. We observed that $\rho_{\pm}$ are supported by $[-R,R]$, and the essential
spectrum can be determined by decomposing into half lines, so if there is spectrum outside $[-R,R]$, it can only consist of discrete eigenvalues.
If we had such a discrete eigenvalue at $E_0$, $|E_0|>R$, then the corresponding eigenfunction $u$ must satisfy $u(0)\not= 0$
because if $u(0)=0$, then $u$ would be in the domain of the half line problems and thus $\rho_{\pm}(\{ E_0 \} )>0$,
contradicting the fact that these measures are supported by $[-R,R]$. However, $u(0)\not=0$ says that $u$ has non-zero
scalar product with $\delta_0$, thus the representing measure of $g(z)=\langle \delta_0, (J-z)^{-1}\delta_0 \rangle$ has
a point mass at $E_0$. This implies that $a_0^2 H(x) = -1/g(x)$ changes its sign at $x=E_0$ (this function is holomorphic near $E_0$,
so this statement makes sense), but we already argued in the first part of this proof that \eqref{3.1} prevents such a sign change.
\end{proof}
It was proved in \cite[Theorem 1.2]{RemDR} that if $J\in\mathcal M_R$ for some $R\ge 2$, then $a_n\ge 1$ for all $n\in\Z$.
Moreover, if $a_{n_0}=1$ for a single $n_0\in\Z$, then $a_n\equiv 1$, $b_n\equiv 0$. This is now an immediate consequence
of Theorem \ref{Tdisc}. Indeed, \eqref{a0} says that $1/a_0^2=1-\sigma_{-2}\le 1$, and we can only have equality here if
$\sigma_{-2}=0$, which forces $\sigma$ to be the zero measure. It's easy to check that this makes $m_{\pm}$ equal to the
half line $m$ functions of the free Jacobi matrix. To obtain the full claim, it now suffices to recall that $\mathcal M_R$ is shift invariant.

It is tempting to try to obtain more information about the coefficients of a $J\in\mathcal M_R$ in this way,
by relating them to the moments of $\sigma$. The following result is probably unimpressive, but it can serve as an illustration.
Also, as we'll discuss after the proof, it is optimal.
Recall that we define $r\in (0,1]$ by the equation $r+1/r=R$.
\begin{Proposition}
\label{P3.11}
If $J\in\mathcal M_R$ is not the free Jacobi matrix, then for all $n\in\Z$, we have that $a_n> 1$ and
\begin{equation}
\label{3.11}
r^2 < \frac{a_{n+1}^2-1}{a_n^2-1} < \frac{1}{r^2} .
\end{equation}
\end{Proposition}
\begin{proof}
The inequality $a_n>1$ was established above; we only need to prove \eqref{3.11}. By comparing \eqref{3.12} with \eqref{m-},
we obtain that
\begin{equation}
\label{3.13}
\sigma_{-2}=1-\frac{1}{a_0^2}, \quad\quad \sigma_0=\frac{a_{-1}^2-1}{a_0^2} .
\end{equation}
Now $r^2< t^{-2} < 1/r^2$ on the support of $\sigma$, hence
\begin{equation}
\label{3.14}
r^2\sigma_0< \sigma_{-2} < \frac{1}{r^2}\sigma_0 .
\end{equation}
Strict inequality would in fact not follow for the the zero measure $\sigma=0$, but that would lead us back to
the free Jacobi matrix, the case that we explicitly excluded.

Now \eqref{3.11}, for $n=-1$, follows by combining \eqref{3.14} with \eqref{3.13}.
We then obtain \eqref{3.11} for arbitrary $n$ by shift invariance.
\end{proof}
The inequalities \eqref{3.11} are indeed sharp, as we pointed out earlier, because they are a rephrasing of \eqref{3.14}, and
we can get arbitrarily close to equality here with measures of the form $\sigma=g\delta_{1/r-\epsilon}$ or $\sigma=g\delta_{r+\epsilon}$.
\section{The continuous case}
We consider Schr\"odinger operators $H=-d^2/dx^2+V(x)$ on $L^2(\R)$, with locally integrable potentials $V$.
We assume limit point case at $\pm\infty$. Then, for $z\in\C^+$, there are unique (up to a constant factor) solutions
$f_{\pm}$ of $-f''+Vf=zf$ that are square integrable near $\pm\infty$. The half line $m$ functions may now be defined as follows:
\begin{equation}
\label{mcont}
m_{\pm}(z) = \pm \frac{f'_{\pm}(0,z)}{f_{\pm}(0,z)}
\end{equation}
These obey the asymptotic formulae
\begin{equation}
\label{mcasymp}
m_{\pm}(z) = \sqrt{-z} + o(1)
\end{equation}
as $|z|\to\infty$ inside a sector $\delta \le \arg z \le \pi-\delta$. See, for example, \cite{Atk,Everitt,GS2,Harris}.

We proceed as in the previous section. We now say that $H\in\mathcal M_R$ if $H$ is reflectionless on $(0,\infty)$ and
$\sigma(H)\subset [-R^2,\infty)$. Occasionally, we will abuse terminology and/or notation and instead say that $V$ is in $\mathcal M_R$.
For $H\in\mathcal M_R$, we again obtain a holomorphic function $M:\Omega\to\C^+$ from Lemma \ref{L2.1},
where now $\Omega = \C^+\cup (0,\infty)\cup \C^-$. Off the real line, $M$ is again given by \eqref{formM}. We use the conformal
map $\varphi:\C^+\to\Omega$, $\varphi(\lambda)=-\lambda^2$ to introduce the Herglotz function $F(\lambda)=M(\varphi(\lambda))$.
We then have the following analog of Theorem \ref{Tdisc}.
\begin{Theorem}
\label{Tcont}
$H\in\mathcal M_R$ if and only if the associated $F$ function is of the form
\begin{equation}
\label{4.3}
F(\lambda) =\lambda + \int\frac{d\sigma(t)}{t-\lambda}
\end{equation}
for some finite Borel measure $\sigma$ on $(-R,R)$ that satisfies
\begin{equation}
\label{4.1}
1+ \int\frac{d\sigma(t)}{t^2-R^2} \ge 0 .
\end{equation}

Moreover, if $H\in\mathcal M_R$, then $V$ is real analytic. More specifically,
$V(x)$ has a holomorphic continuation $V(z)$ to the strip $|\textrm{\rm Im}\, z|<1/R$.
\end{Theorem}
As in the discrete case, this establishes a one-to-one correspondence between Schr\"odinger operators $H\in\mathcal M_R$ and
measures $\sigma$ on $(-R,R)$ satisfying \eqref{4.1}. Also as before, if we are not interested in the value of $R$, then we can say
that Theorem \ref{Tcont} provides us with a one-to-one correspondence between Schr\"odinger operators $H$ that are reflectionless
on $(0,\infty)$ and bounded below and compactly supported measures $\sigma$.
\begin{proof}
It is again straightforward to check that given an $H\in\mathcal M_R$, the corresponding $F$ has such a representation.
The general Herglotz representation of $F$ reads
\[
F(\lambda) = A + B\lambda + \int \left(\frac{1}{t-\lambda}-\frac{t}{t^2+1}\right)\, d\sigma(t) .
\]
Now \eqref{mcasymp} immediately shows that $B=1$ here. Moreover, $m_{\pm}(z)$ have holomorphic continuations through $(-\infty,-R^2)$.
Since $\R\setminus [-R,R]$ gets mapped to this set under $\varphi$, it follows that $\sigma$ is supported by $[-R, R]$,
as claimed (point masses at the end points will be prevented by \eqref{4.1}). We can again split off the second term from the integral and absorb
it by $A$. The redefined $A$ must then satisfy $A=0$, by \eqref{mcasymp}. Thus \eqref{4.3} holds.

To obtain \eqref{4.1}, we again consider $H=m_+ +m_-$ and $h(\lambda)=H(\varphi(\lambda))$, for $\lambda\in\C^+$,
$\textrm{Re}\,\lambda < 0$. This function has a holomorphic continuation through the imaginary axis, and for $\lambda=iy$, $y>0$,
we have that $\overline{\lambda}=-\lambda$, thus for these $\lambda$, it follows that
\begin{equation}
\label{4.6}
h(\lambda)  =F(\lambda)-\overline{F(\lambda)}
 = 2\lambda \left( 1 + \int\frac{d\sigma(t)}{t^2-\lambda^2} \right) .
\end{equation}
We conclude the argument as in the discrete case: By analyticity, \eqref{4.6} holds for all $\lambda$ in the second quadrant $Q_2$.
The function $h(\lambda)$ (more precisely: its boundary value as $\varphi(\lambda)\to x\in\R$, $x<-R^2$) must be negative for
all $\lambda\in\R$ with $\lambda<-R$, and the factor $2\lambda$ already has the correct sign, so the expression in parentheses must be positive.
By monotone convergence, when $\lambda$ increases to $-R$, the integrals $\int\frac{d\sigma}{\lambda^2-t^2}$ approach
$\int\frac{d\sigma}{R^2-t^2}$ and they increase strictly. Therefore, the condition that the last factor from \eqref{4.6} is positive
for all $\lambda<-R$ is equivalent to \eqref{4.1}.

Conversely, if a measure $\sigma$ on $(-R,R)$ satisfying \eqref{4.1} is given, define $F$ by \eqref{4.3} and then
\begin{align}
\label{4.41a}
m_+(\varphi(\lambda))&=F(\lambda)&(\lambda\in Q_2)\\
m_-(\varphi(\lambda))& = -\overline{F(\overline{\lambda})} & (\lambda\in Q_4) ;
\label{4.41b}
\end{align}
here, $Q_j\subset\C$ denotes the (open) $j$th quadrant.
By construction, this pair of Herglotz functions satisfies \eqref{defrefl} on $S=(0,\infty)$. We must show that $m_{\pm}$ are the half line
$m$ functions of a Schr\"odinger operator $H$. We thus need an inverse spectral theory result for Schr\"odinger operators that lets us
verify this claim. We will
refer to the classical Gelfand-Levitan theory; the version we will use is taken from \cite{RemdB}. Note that since we are dealing with limit point
operators here and since it is clear that $m_+(z)=\sqrt{-z}+o(1)$ as $|z|\to\infty$ inside suitable sectors for the $m_+$ just defined, we may
state the results of the discussion of \cite[Sect.~19]{RemdB} as follows (for convenience, we focus on the right half line for now):
Let $d\rho_0(x)=(1/\pi)\chi_{(0,\infty)}(x)\sqrt{x}\, dx$ be the half line spectral measure for zero potential. Consider the signed measure
$\nu=\rho_+ -\rho_0$, where $\rho_+$ is the measure associated with $m_+$. Then $m_+$ is the $m$ function of some half line
Schr\"odinger operator (with locally integrable potential) if and only if $\rho_+$ satisfies the following two conditions:
\begin{enumerate}
\item If $f\in L^2(0,L)$ for some $L>0$ and $\int |F|^2\, d\rho_+=0$, with $F(x) = \int f(t)\frac{\sin{t\sqrt{x}}}{\sqrt{x}}\, dt$, then $f=0$.
\item It is possible to define a distribution $\phi$ by
\begin{equation}
\label{4.11}
\phi(t) = \int \frac{\sin{t\sqrt{x}}}{\sqrt{x}}\, d\nu(x) .
\end{equation}
Moreover, $\phi$ is a locally integrable function.
\end{enumerate}
More explicitly, what (2) is asking for is the following: If $g\in C_0^{\infty}(\R)$, then
\begin{equation}
\label{conv}
\int d|\nu|(x) \left| \int dt\, g(t)\frac{\sin{t\sqrt{x}}}{\sqrt{x}} \right| < \infty
\end{equation}
and there is a locally integrable function $\phi$ so that for all $g\in C_0^{\infty}(\R)$, we have that
\begin{equation}
\label{phi}
\int d\nu(x) \int dt\, g(t)\frac{\sin{t\sqrt{x}}}{\sqrt{x}} = \int \phi(t)g(t)\, dt .
\end{equation}

Let us now check these conditions for the $m_+$ (or rather, $\rho_+$) defined above. To learn more about $\rho_+$, we have to analyze
the boundary values of $m_+(z)$ as $z$ approaches the real line; this corresponds to letting $\lambda\in Q_2$ approach either the
negative real axis or the positive imaginary axis. We find that
\[
d\rho_+(x) = d\mu(x) + \frac{1}{\pi}\chi_{(0,\infty)}(x) \textrm{Im}\,F(ix^{1/2})\, dx ,
\]
and here $\mu$ is a finite measure, supported by $[-R^2,0]$. In particular, $\rho_+$ is equivalent to Lebesgue measure on $(0,\infty)$,
so condition (1) holds trivially. As for condition (2), this definitely holds for compactly supported $\nu$; the locally integrable function
$\phi$ can then simply be obtained by taking \eqref{4.11} at face value. Also, to establish (2) for a sum of measures,
it clearly suffices to verify this condition for the individual summands separately.

So by splitting off a compactly supported part, we can now focus on
\[
d\nu_1(x) = \frac{1}{\pi} \chi_{(1,\infty)}(x) \left( \textrm{Im}\,F(ix^{1/2}) - x^{1/2} \right)\, dx .
\]
Observe that near infinity, $F(\lambda)=\lambda-\sigma_0\lambda^{-1} + O(\lambda^{-2})$, thus
\[
d\nu_1(x) = c\chi_{(1,\infty)}(x)x^{-1/2}\, dx + f(x)\, dx
\]
where the density $f\in C([1,\infty))$ satisfies $f(x)=O(x^{-1})$. It is clear that this decay is fast enough to give (2) for
this part of $\nu_1$; we will again end up interpreting \eqref{4.11} as a classical integral. By again splitting off a compactly supported part, we thus
see that it now suffices to verify (2) for the measure
\[
d\nu_2(x)=\chi_{(0,\infty)}(x) x^{-1/2}\, dx
\]
Clearly, \eqref{conv} holds. It also clear that the left-hand side of \eqref{phi} does define a distribution, and in fact a
tempered distribution. We now compute its Fourier transform. So apply the left-hand side to the Fourier transform $\widehat{g}$
of a test function $g$. We obtain that
\begin{align*}
\int d\nu_2(x) \int dt\, \widehat{g}(t)\frac{\sin{t\sqrt{x}}}{\sqrt{x}}  & =
-i\sqrt{\frac{\pi}{2}}\int_0^{\infty} \left( g(x^{1/2})-g(-x^{1/2})\right) \frac{dx}{x}\\
& = -i\sqrt{2\pi} \int_0^{\infty} (g(s)-g(-s))\frac{ds}{s} .
\end{align*}
It is easy to verify that this last integral equals $(\textrm{PV}(1/s), g)$, where the principal value distribution is defined as follows:
\[
\left( \textrm{PV}\left(\frac{1}{s}\right) , g\right) = \lim_{\delta\to 0+} \int_{|s|>\delta} \frac{g(s)}{s}\, ds
\]
Since $\textrm{PV}(1/s)$ is the Fourier transform of $i(\pi/2)^{1/2}\,\textrm{sgn}(t)$, we now see that $\phi_2(t) =\pi\,\textrm{sgn}(t)$, which is a locally
integrable function, as claimed.

Of course, one can give an analogous discussion for the left half line and $m_-$. So, to conclude the proof of the first part of the theorem,
we must show that the Schr\"odinger operator obtained above has spectrum contained in $[-R^2,\infty)$. This can be done
by the same arguments as in the discrete case:
Clearly, by the decomposition method for $\sigma_{ess}$, as $\rho_{\pm}$ are supported by this set, there is no \textit{essential }spectrum outside
$[-R^2,\infty)$. If we had a discrete eigenvalue $E_0<-R^2$, then the corresponding eigenfunction $u$ must satisfy $u(0)\not=0$ because otherwise
$\rho_{\pm}(\{ E_0\})>0$, but we already know that this is not the case. It then follows from the standard construction
of a spectral representation of the whole problem (see, for example, \cite[Sect.~9.5]{CodLev})
that $\rho(\{ E_0\} )>0$, where $\rho$ denotes the
measure associated with the Green function $g=-1/(m_++m_-)$. This implies that $H=m_++m_-$ changes its sign at $E_0$, but this
is incompatible with \eqref{4.1}: Recall that we in fact specifically formulated
\eqref{4.1} as the condition that would guarantee that $H$ is negative throughout $(-\infty,-R^2)$.

We now move on to the last part of the proof, which discusses the real analyticity of $V\in\mathcal M_R$. We will obtain this property from
the Riccati equation that is satisfied by $m_+$, together with a Taylor expansion about infinity. This part of the argument
essentially follows the treatment of \cite{Mar}.

Given a potential $V\in\mathcal M_R$, let
\[
p(w)= \frac{1}{w} - F\left(\frac{1}{w} \right) .
\]
We originally define this function for $w\in Q_3$; this choice makes sure that $F(1/w)=m_+(-1/w^2)$. However, it is also clear that
$p$ has a holomorphic continuation to a neighborhood of $w=0$. The corresponding Taylor expansion may be found from \eqref{4.3}:
\begin{equation}
\label{4.16}
p(w) = \sum_{n=0}^{\infty} \sigma_n w^{n+1} ,
\end{equation}
where we again write $\sigma_n=\int t^n\, d\sigma(t)$. We now claim that for $n\ge 0$,
\begin{equation}
\label{estmom}
|\sigma_n|\le R^{n+2} .
\end{equation}
To prove this, observe that obviously $|\sigma_n|\le \sigma_0 R^n$, since $\sigma$ is supported by $(-R,R)$. Now condition
\eqref{4.1} implies that $\sigma_0\le R^2$,
so we obtain \eqref{estmom}.
It follows that \eqref{4.16} converges at least on
$|w|<1/R$.

We now consider the shifted potentials $V_x(t)=V(x+t)$ and the associated data $p(x,w)$, $\sigma_n(x)$.
Since $\mathcal M_R$ was defined in terms of shift invariant conditions, $V_x$ will also be in $\mathcal M_R$ for all $x$.

From \eqref{mcont}, we obtain that (for $w\in Q_3$)
\begin{equation}
\label{ricd}
\frac{dp}{dx} = -V(x) + p^2(x,w) - \frac{2}{w}\, p(x,w) .
\end{equation}
We now temporarily work with the integrated form of this equation. We may then replace every occurrence of $p$ by its expansion \eqref{4.16};
this we can do for $|w|<1/R$. The interchange of series and integration in the resulting expressions is easily justified: The coefficients
$\sigma_n(x)$ are measurable (they can be obtained as derivatives with respect to $w$, so are pointwise limits of measurable functions),
and \eqref{estmom} gives uniform (in $x$) control, so dominated convergence applies. This produces
\begin{align*}
\sum_{n\ge 0} (\sigma_n(x)-\sigma_n(0))w^{n+1} & = -\int_0^x V(t)\, dt + \sum_{j,k\ge 0} w^{j+k+2}\int_0^x \sigma_j(t)\sigma_k(t)\, dt \\
& \quad\quad\quad - 2 \sum_{n\ge 0} w^n \int_0^x \sigma_n(t)\, dt .
\end{align*}
This was originally derived for $w\in Q_3$, $|w|<1/R$, but since both sides are holomorphic in $w$, the equation holds for all $|w|<1/R$.

We can now compare coefficients in these convergent power series. Starting with the constant terms, this gives that $\int_0^x V\, dt +
2\int_0^x \sigma_0\, dt=0$ or, by differentiation,
\begin{equation}
\label{sigma0}
V(x) = -2\sigma_0(x)
\end{equation}
for almost every $x$. Since $V$ may be redefined in an arbitrary way on a null set, we can assume that \eqref{sigma0} holds for all $x\in\R$.
(Of course, $\sigma_0(x)$ is well defined pointwise, for any given $x$, independently of the representative of $V$ chosen,
as the zeroth moment of the measure $d\sigma(x,\cdot)$ that represents the
$F$ function of $V_x$.)

Next, we obtain that
\[
\sigma_0(x)-\sigma_0(0) = -2 \int_0^x \sigma_1(t)\, dt .
\]
This shows that $\sigma_0$ is in fact absolutely continuous, and $\sigma'_0=-2\sigma_1$. Proceeding in this way, we see inductively
that $\sigma_n(x)$ is an absolutely continuous function for arbitrary $n\ge 0$. Moreover, since the derivatives $\sigma'_n$ are built from finitely
many other functions $\sigma_j$, they are bounded functions by \eqref{estmom}. We have a crude preliminary bound of the form
$|\sigma'_n(x)|\le CnR^n$.
This allows us to differentiate the series \eqref{4.16} (with respect to $x$) term by term, for $|w|<1/R$. We then return to the differential version
\eqref{ricd} of the Riccati equation.
By again comparing coefficients of power series, we finally arrive at the following recursion formulae:
\begin{align}
\nonumber
V(x) & = -2\sigma_0(x)\\
\label{recurs}
\sigma'_0(x) & = -2\sigma_1(x)\\
\nonumber
\sigma'_n(x) & = -2\sigma_{n+1}(x) + \sum_{j=0}^{n-1} \sigma_j(x)\sigma_{n-1-j}(x) \quad\quad\quad (n\ge 1)
\end{align}
Formally, this could have been obtained very quickly from \eqref{ricd}, but initially we did not know that the
$\sigma_n(x)$ are differentiable, so we had to be more circumspect. We now use this recursion to obtain more detailed
information about the $\sigma_n(x)$.
\begin{Lemma}
\label{L4.6}
The moments $\sigma_n(x)$ satisfy $\sigma_n\in C^{\infty}(\R)$ and
\begin{equation}
\label{estsigma}
\left| \sigma_n^{(p)}(x) \right| \le R^{n+p+2} \frac{(n+1+p)!}{(n+1)!} .
\end{equation}
\end{Lemma}
Assuming Lemma \ref{L4.6}, we can finish the proof of Theorem \ref{Tcont} very quickly. By \eqref{sigma0}, the Lemma in particular
says that $V\in C^{\infty}$. Now \eqref{estsigma}, for $n=0$ and general $p\ge 0$, may be used to confirm that the Taylor series
of $V(x)$ about an arbitrary $x_0\in\R$ has radius of convergence $\ge 1/R$. We can then refer to the same estimates and one of the
standard bounds on the remainder to see that this Taylor series converges to $V(x)$ on $(x_0-1/R,x_0+1/R)$. Since the strip
$|\textrm{Im}\, z| < 1/R$ is simply connected, this shows that $V$ has a holomorphic continuation to the whole strip.
\end{proof}
\begin{proof}[Proof of Lemma \ref{L4.6}]
We already know that $\sigma_n\in C^1$, so
the first claim follows from \eqref{recurs}, by an obvious inductive argument. We prove \eqref{estsigma} by induction on $p$.
For $p=0$, this is just \eqref{estmom}. Now assume that \eqref{estsigma} holds for $0,1,\ldots ,p$ and all $n\ge 0$. We wish to establish
the same estimates for $p+1$ and all $n\ge 0$. We will explicitly discuss only the case $n\ge 1$; $n=0$ is similar, but much easier.
The Leibniz rule says that
\[
\frac{d^p}{dx^p}\left( \sigma_j\sigma_{n-1-j} \right) = \sum_{k=0}^p \binom{p}{k} \sigma_j^{(k)}\sigma_{n-1-j}^{(p-k)} ,
\]
so from \eqref{recurs} and the induction hypothesis we obtain that
\begin{align*}
\left| \sigma_n^{(p+1)}\right| & \le 2R^{n+p+3} \frac{(n+2+p)!}{(n+2)!}\: +\\
& \quad\quad\quad \sum_{j=0}^{n-1}\sum_{k=0}^p \binom{p}{k}
R^{n+p+3} \frac{(j+1+k)!}{(j+1)!}\frac{(n-j+p-k)!}{(n-j)!} .
\end{align*}
As observed in \cite[pg.~293]{Mar}, the sum over $k$ can be evaluated: we have that
\begin{equation}
\label{idmar}
\sum_{k=0}^p \binom{p}{k} \frac{(j+1+k)!}{(j+1)!}\frac{(n-j+p-k)!}{(n-j)!} = \frac{(n+p+2)!}{(n+2)!} 
\end{equation}
(we'll return to this formula in a moment).
Since the answer provided in \eqref{idmar} is independent of $j$, we can now also sum over this index. This gives
\begin{align*}
\left| \sigma_n^{(p+1)}\right| & \le 2R^{n+p+3} \frac{(n+2+p)!}{(n+2)!} + nR^{n+p+3} \frac{(n+p+2)!}{(n+2)!}\\
& = R^{n+p+3}
\frac{(n+p+2)!}{(n+1)!} ,
\end{align*}
as desired.

It remains to verify \eqref{idmar}. This can be rephrased: we must show that
\[
\sum_{k=0}^p \binom{N_1+k}{k}\binom{N_2-k}{p-k} = \binom{N_1+N_2+1}{p} ,
\]
for integers $N_1\ge 1$, $N_2\ge p$. It's not hard to convince oneself that the left-hand side can be given the same combinatorial
interpretation as the right-hand side (choose $p$ objects from a collection of $N_1+N_2+1$), so this identity holds.
\end{proof}
\section{Proof of Theorems \ref{T1.1}, \ref{T1.1b}, \ref{T1.1c}}
This will depend on material from \cite{Remcont}. We will give a quick review, but will refer the reader to \cite{Remcont}
for some of the more technical details.

The key tool is \cite[Theorem 3]{Remcont}, which says that if $V$ satisfies Hypothesis \ref{H1.1}, then any $\omega$ limit
point $W=\lim S_{x_n}V$ (that is, any such limit for a sequence $x_n\to\infty$)
under the shift map $(S_xV)(t)\equiv V(x+t)$ must be reflectionless on $(0,\infty)$. These limits
are taken inside a certain metric space $(\mathcal V^C, d)$ of whole line potentials. In fact, $\mathcal V^C$ is a space of signed
measures on $\R$, and locally integrable potentials $U$ are interpreted as the measures $U(x)\, dx$. However, for our purposes here,
measures can be avoided. This is so because the measure analog of the space $\mathcal M_R$ contains no new members: all
such measures will be (real analytic) functions anyway. The key fact here is the observation that we will still have \eqref{mcasymp}
for a Schr\"odinger operator $-d^2/dx^2+\mu$ with a measure, as long as $\mu(\{ 0 \})=0$. This follows from the standard
proofs of \eqref{mcasymp}, suitably adjusted. See also \cite[Lemma 5.1]{BARem}. If $\mu(\{0\})\not= 0$, then we can shift
and instead consider $S_{x_0}\mu$ for an $x_0$ with $\mu(\{ x_0\} )=0$.
With \eqref{mcasymp} in place, we can then follow the development given in Sect.~4 to confirm that an operator $-d^2/dx^2+\mu\in
\mathcal M_R$ still has an $F$ function of the form described in Theorem \ref{Tcont}, so no new operators are obtained.

The metric $d$ is described in detail in \cite{Remcont}; here, we will be satisfied with a non-technical description.
For our purposes, the following properties are important. First of all, convergence to a $W$ with respect to $d$ is equivalent to the condition that
\begin{equation}
\label{4.21}
\int W(t)\varphi(t)\, dt = \lim_{n\to\infty} \int V(x_n+t)\varphi(t)\, dt
\end{equation}
for all continuous, compactly supported test functions $\varphi$. (Only limit points $W\in\mathcal M_R$
will occur in our situation, so we may assume here that $W$ is continuous, say.)
Second, the spaces $(\mathcal V^C,d)$ are compact. Since also $\{S_xV\}\subset\mathcal V^C$, this means that we can always pass to
convergent subsequences of shifted versions of the original potential. Similarly, the spaces $\mathcal M_R$ are compact if endowed with the
same metric $d$.

Finally, it's easy to see that limit points $W$ cannot have spectrum outside the (in fact: essential) spectrum
of $H_+$ \cite[Proposition 1]{Remcont}. Thus they will lie in $\mathcal M_R$ if we take $R\ge 0$ so large that $H_+$ has no
(essential) spectrum below $-R^2$.

The second crucial ingredient to all three proofs is the following immediate consequence of \eqref{sigma0}: any $W\in\mathcal M_R$
satisfies $W(x)\le 0$ for all $x\in\R$. Moreover, if $W(x_0)=0$ for a
single $x_0\in\R$, then $W\equiv 0$. This follows as in the discrete case because $W(x_0)=0$ forces $\sigma$ (for $x_0$)
to be the zero measure, and this makes $m_{\pm}$ equal to the $m$ functions for zero potential.
\begin{proof}[Proof of Theorem \ref{T1.1}]
If the statement of the Theorem didn't hold, then we could find a sequence $x_n\to\infty$ so that $S_{x_n}V\to W$ (using compactness)
and \eqref{1.6} along that sequence converges to some $a>0$, for some test function $\varphi$. But then \eqref{4.21} forces $W$ to
be positive somewhere.
\end{proof}
\begin{proof}[Proof of Theorem \ref{T1.1b}]
This is similar. The extra assumption on $V$, if combined with \eqref{4.21}, makes sure that every limit point $W$ is non-negative somewhere.
As explained above, this implies that $W\equiv 0$. In other words, the zero potential is the only possible limit point.
\end{proof}
\begin{proof}[Proof of Theorem \ref{T1.1c}]
This will again follow from the same ideas. Fix a test function $\psi\ge 0$, $\int\psi = 1$ that is supported by $(-d,d)$.
We claim that we can find $\delta>0$ such that if $W\in\mathcal M_R$ satisfies $\int W\psi > -2\delta$ (recall that $W\le 0$, so
$\int W\psi\le 0$), then
\begin{equation}
\label{4.51}
\left| \int W(t)\varphi_j(t)\, dt \right| < \epsilon \quad\quad (j=1,\ldots, N) .
\end{equation}
This is a consequence of the compactness of $\mathcal M_R$: If our claim was wrong,
then we could find a sequence $W_n\to W$, $W_n,W\in\mathcal M_R$ so that $\int W_n\psi\to 0$, but \eqref{4.51} fails for all $W_n$.
But then $\int W\psi =0$, hence $W= 0$ on the support of $\psi$, hence $W\equiv 0$. Thus \eqref{4.51} could not fail for all $W_n$ in this situation.
Our claim was correct. We can and will also insist here that $\delta\le\epsilon$.

With this preparation out of the way, use compactness again to find an $x_0$ with the property that for each $x\ge x_0$,
there is a limit point $W\in\mathcal M_R$, which will depend on $x$, so that
\[
\left| \int (W(t)-V(x+t))\theta(t)\, dt \right| < \delta
\]
for the test functions $\theta=\psi$ and $\theta=\varphi_j$.
Now if $V\ge -\delta$ on $(x-d,x+d)$, then $\int V(x+t)\psi(t)\, dt \ge -\delta$, thus $\int W\psi > -2\delta$, so \eqref{4.51} applies and it follows that
\[
\left| \int V(x+t)\varphi_j(t)\, dt \right| < \delta + \epsilon \le 2\epsilon ,
\]
as desired.
\end{proof}
\section{Proof of Theorem \ref{T1.2}}
(a) Recall how we obtained the conditions on $F$ and $\sigma$ for a $J\in\mathcal M_R$ in the proof of Theorem \ref{Tdisc}:
Essentially, we had to make sure that the behavior of $F$ as $\lambda\to 0$ and $|\lambda|\to\infty$ is consistent with the known asymptotics
of $m_+(z)$ and $m_-(z)$, respectively, as $|z|\to\infty$. If we only want $m_+$ to be the $m$ function of a (half line) Jacobi matrix,
but not $m_-$, then we only need to make sure that the asymptotics of $F$ as $\lambda\to 0$ come out right.

To obtain such an example, let's just take $\sigma=\delta_1$, so
\begin{equation}
\label{6.9}
F(\lambda)= -1 + \frac{1}{1-\lambda} .
\end{equation}
As $F$ approaches a limit as $|\lambda|\to\infty$, this is clearly not the $F$ function of a whole line Jacobi matrix.
(So the point really was to choose a $\sigma$ with $\sigma_{-2}=1$, to destroy the required asymptotics at large $\lambda$.)
However, \eqref{6.9} will yield an $m$ function $m_+$ of a (positive) half line Jacobi matrix $J_+$ via \eqref{m++}.
This follows as in the proof of Theorem \ref{Tdisc}; notice that \eqref{3.1} was not used in this part of the argument.
Also, by construction, this $m_+$ will satisfy \eqref{defrefl} on $(-2,2)$, for the companion Herglotz function $m_-$ that is also extracted from $F$,
via \eqref{m--}.

So we have already proved Theorem \ref{T1.2}(a).
However, it is also interesting to work things out somewhat more explicitly.
We can find $m_+(z)$ most conveniently by using the material from Sect.~2. Notice that \eqref{3.5} becomes
\[
h(\lambda)=h_0(\lambda)\frac{1}{(1-\lambda)(1-1/\lambda)} ,
\]
hence
\begin{equation}
\label{H}
H(z) = \frac{H_0(z)}{z+2} ,
\end{equation}
where $H_0(z)=\sqrt{z^2-4}$ is the $H$ function of the free Jacobi matrix. Now \eqref{2.1}, specialized to the case at hand, says that
$m_+ = A_+ + (1/2)H$. Here we use the fact that the measure $\rho$ associated with $H$ is supported by $(-2,2)$, as we read off from
\eqref{H}; there is no point mass at $-2$ because $H_0$ contains the factor $(z+2)^{1/2}$.
We also know that $m_+(iy)\to 0$ as $y\to\infty$ (because $F(\lambda)\to 0$ as $\lambda\to 0$), and this implies that $A_+=-1/2$. Thus
\[
m_+(z) = \frac{1}{2} \left( \sqrt{\frac{z-2}{z+2}} - 1 \right) ;
\]
of course, the square root must be chosen so that $m_+$ becomes a Herglotz function.
With this explicit formula, we can confirm one more time that $m_+$ is the $m$ function of a Jacobi matrix $J_+$. The associated measure
can also be read off:
\[
d\rho_+(x) = \frac{1}{2\pi}\chi_{(-2,2)}(x) \sqrt{\frac{2-x}{2+x}}\, dx
\]
In particular, we can now confirm the additional claim that $\sigma(J_+)=[-2,2]$ that was made earlier, in Sect.~1.

It is instructive to obtain this example as a limit of measures $\sigma_{\epsilon} = (1-\epsilon)\delta_1$. For $\epsilon>0$ (and small),
these measures obey \eqref{3.1}, so are admissible in the sense of Theorem \ref{Tdisc}. The $F$ function is given by
\[
F_{\epsilon}(\lambda) = -1+\epsilon + \epsilon\lambda + \frac{1-\epsilon}{1-\lambda} .
\]
A similar analysis can be given. The associated
Jacobi matrices $J_{\epsilon}$ have an eigenvalue at $E_{\epsilon}=-1-1/\epsilon$ and no other spectrum outside $[-2,2]$; of course,
they are reflectionless on $(-2,2)$. (Operators in $\mathcal M_R$ with only discrete spectrum outside $[-2,2]$ are usually called \textit{solitons.})
So our example shows the following: There is a sequence of solitons $J_{\epsilon}$ so that the half line restrictions $(J_{\epsilon})_+$ converge,
in the strong operator topology, to our $J_+$ from above. The unrestricted whole line operators $J_{\epsilon}$ do not converge, of course;
their operator norms form an unbounded sequence. In fact, \eqref{a0} informs us that $a_0=\epsilon^{-1/2}$, so this is already divergent.

(b) This is very similar, but somewhat more tedious from a technical point of view. Since we already went through similar arguments
in the proof of Theorem \ref{Tcont}, we will be satisfied with a sketch. Let
\[
d\sigma(t) = \chi_{(1,\infty)}(t)e^{-t}\, dt ;
\]
as the discussion we are about to give will make clear, only certain general features of this measure matter, not its precise form.
Note that a compactly supported $\sigma$ can not produce an example of the desired type, as observed above, after Theorem \ref{Tcont}.
As in part (a), the basic idea is to leave the asymptotics of $m_+$ essentially untouched while seriously upsetting those of $m_-$.
Indeed, if we now define $F$ by \eqref{4.3} and then $m_{\pm}$ by \eqref{4.41a}, \eqref{4.41b} and extract the corresponding measures $\rho_{\pm}$,
then we find that
\begin{align*}
d\rho_+(x) & = d\rho_0(x) + \chi_{(0,\infty)}(x)f(x)\, dx \\
d\rho_-(x) & = d\rho_+(x) + \chi_{(-\infty,-1)}(x)e^{-|x|^{1/2}}\, dx ,
\end{align*}
with a density $f$ that again satisfies $f(x)=cx^{-1/2}+O(x^{-1})$ as $x\to\infty$. Exactly this situation was discussed in the proof of
Theorem \ref{Tcont}: such a $\rho_+$ satisfies conditions (1), (2) from the Gelfand-Levitan theory, and since also $m_+(z)=\sqrt{-z}+o(1)$
for large $|z|$, it follows that $m_+$ is the $m$ function of a half line Schr\"odinger operator $H_+$.
Notice also that $\rho_+$ is supported by $(0,\infty)$, so indeed $\sigma(H_+)=[0,\infty)$.

To finish the proof, we show that $\nu=\rho_--\rho_0$ does not satisfy condition (2). Now we just saw that $\rho_+-\rho_0$ does
define a locally integrable function via \eqref{4.11} (interpreted in distributional sense), so this will follow if we can show that
the formal expression
\begin{equation}
\label{5.1}
\int_{-\infty}^{-1} \frac{\sin t\sqrt{x}}{\sqrt{x}} e^{-|x|^{1/2}}\, dx
\end{equation}
does \textit{not }define a locally integrable function. In fact, it is almost immediate that with the interpretation given above,
\eqref{5.1} does not even define a distribution: Since $x^{-1/2}\sin tx^{1/2} = |x|^{-1/2}\sinh t|x|^{1/2}$ for $x<0$, it is clear
that \eqref{conv} diverges for any test function $g\ge 0$, $g\not\equiv 0$ whose support lies to the right of $1$.

\end{document}